\documentclass[11pt]{article}

\usepackage[a4paper,margin=2.6cm]{geometry}
\usepackage{amsmath,amssymb,amsthm,mathtools}
\usepackage{bm}
\usepackage{enumitem}
\usepackage{hyperref}
\hypersetup{colorlinks=true,linkcolor=blue,citecolor=blue,urlcolor=blue}
\newtheorem{definition}{Definition}[section]
\newtheorem{theorem}{Theorem}[section]
\newtheorem{lemma}[theorem]{Lemma}

\newtheorem{claim}{Claim}
\newtheorem{case}{Case}

\newtheorem{Problem}[theorem]{Problem}
\newcommand{\ol}[1]{\overline{#1}}
\newcommand{\vv}{\vee}

\newcommand{\E}{\mathrm E}

\title{Spectral Extremal Graphs without a $K_k$-Factor}
\author{{Cunxiang Duan$^{a,b}$, Tingting Han$^{c}$, Lin-Peng Zhang$^{d,\dagger}$ }\\
\\{\small $^{a}$School of Mathematics, East China University of Science and Technology, Shanghai 200237, P.R. China}\\
{\small $^{b}$School of Mathematics, Taiyuan University of Technology, Taiyuan, Shanxi 030024, P.R. China}\\
{\small $^{c}$School of Mathematical Sciences, Shanxi Normal University, Taiyuan, Shanxi 030031, P.R. China}\\
{\small $^{d}$School of Mathematics, Shandong University, Jinan, Shandong 250100, P.R. China}}
\date{}

\begin{document}
\maketitle
\footnotetext[1]{*C.X. Duan is supported by the National Natural Science Foundation of China (No.~1230145), the Natural Science Foundation of Shanxi Province (No.~202303021212025), and the Shanxi Scholarship Council of China (No.~2025-080). T. Han is supported by the National Natural Science Foundation of China (No.~12601682), and the Shanxi Scholarship Council of China (No.~2026-137). L.-P. Zhang is supported by the China Postdoctoral Science Foundation (No. 2025M783120), and Shandong Postdoctoral Science Foundation (No. SDZZ-ZR-202501340).}
\footnotetext[2]{$\dagger$Corresponding author; Email addresses: cxduanmath@163.com (C. Duan), hantingting@sxnu.edu.cn (T. Han),  lpzhangmath@163.com (L. Zhang).}
\begin{abstract}
Let $k\ge 3$ and let $n=km$. A $K_k$-factor in an $n$-vertex graph is a collection of
$m$ vertex-disjoint copies of $K_k$ that covers the entire vertex set. We determine the
maximum adjacency spectral radius of an $n$-vertex graph containing no $K_k$-factor
when $m\ge 2k-1$. More precisely, we prove that every such graph $G$ satisfies
\[
  \rho(G)\le \rho(H_{n,k}),
  \qquad
  H_{n,k}=K_{k-2}\vee\bigl(K_{n-k+1}\cup K_1\bigr),
\]
with equality if and only if $G\cong H_{n,k}$. Equivalently, the unique extremal graph
is obtained from $K_{n-1}$ by adding one vertex adjacent to exactly $k-2$ vertices of
the clique. Our proof combines a decomposition lemma for sparse complements, derived
from the Hajnal--Szemer\'edi theorem, with the Motzkin--Straus inequality and spectral
estimates based on quotient matrices and the Rayleigh quotient.
\end{abstract}


\section{Introduction}
Throughout this paper, all graphs are finite, simple, and undirected. For standard
terminology and notation in spectral graph theory, we refer to
\cite{cve1980}.

Graph factors are among the central spanning structures in extremal graph theory. Given
graphs $G$ and $H$, an $H$-tiling in $G$ is a collection of pairwise vertex-disjoint
copies of $H$. An $H$-factor is an $H$-tiling that covers every vertex of $G$. Thus,
when $H=K_k$ and $n=km$, a $K_k$-factor consists of $m$ vertex-disjoint copies of
$K_k$ whose union is $V(G)$. A fundamental theorem of Hajnal and Szemer\'edi
\cite{HS1970} states that every $n$-vertex graph $G$, with $k\mid n$,
contains a $K_k$-factor whenever
\[
  \delta(G)\ge \left(1-\frac1k\right)n.
\]
This degree threshold is best possible. The theorem is a cornerstone of graph-factor
theory: it shows that sufficiently strong and uniformly distributed local density forces
a perfect packing of cliques.

A natural spectral analogue is to replace the minimum-degree condition by a condition on
the adjacency spectral radius. The spectral radius $\rho(G)$ captures global edge density
and, more importantly, how strongly that density is concentrated. This makes the spectral
problem substantially different from its minimum-degree counterpart. A graph can have a
very large spectral radius while still containing a small set of vertices that creates a local
obstruction to every spanning clique packing. Consequently, the relevant extremal question
is not merely how dense a $K_k$-factor-free graph can be, but how efficiently the obstruction
to a $K_k$-factor can be localized while preserving a large spectral radius.

More generally, spectral Tur\'an-type problems ask for the largest spectral radius of a graph
that avoids a prescribed subgraph or spanning structure. For spanning forbidden graphs,
the problem is particularly delicate because the obstruction may involve only a small part
of the vertex set, whereas the spectral radius reflects the structure of the whole graph, we refer to \cite{FanLin2024,FanLinLu2023,FanLinZhu2026}.

The Hamilton cycle is perhaps the most prominent example of a spanning structure studied
from this perspective. Fiedler and Nikiforov \cite{FiedlerNikiforov2010} proved sharp
spectral conditions for the existence of Hamilton paths and Hamilton cycles. In extremal
language, their Hamilton-cycle result implies that, for $n\ge 5$, the unique graph of maximum
spectral radius among all $n$-vertex non-Hamiltonian graphs is
\[
  K_1\vee\bigl(K_1\cup K_{n-2}\bigr).
\]
This result was followed by several refinements in which additional structural parameters
were taken into account. In particular, Li and Ning \cite{LiNing2016} determined, for
sufficiently large order, the maximum spectral radius of non-Hamiltonian graphs with a
prescribed lower bound on the minimum degree, thereby obtaining spectral analogues of
classical theorems of Erd\H{o}s and Moon-Moser. Ge and Ning
\cite{GeNing2020} treated the almost-spanning cycle $C_{n-1}$, while Li and Ning
\cite{LiNing2023} subsequently developed stability results and spectral conditions for
cycles whose lengths are close to $n$. These results show that, even for long cycles, the
spectral extremal graph often coincides with the corresponding edge-extremal graph.

Spectral methods have also been applied to degree-constrained spanning subgraphs. A perfect
matching is a $K_2$-factor, and O \cite{O2021} obtained a sharp adjacency-spectral-radius
condition guaranteeing its existence. More generally, an $[a,b]$-factor is a spanning
subgraph $F$ satisfying $a\le d_F(v)\le b$ for every vertex $v$. Fan, Lin, and Lu
\cite{FanLinLu2022} established spectral conditions for odd $[1,b]$-factors and for
$[a,b]$-factors, confirming a conjecture of Cho, Hyun, O, and Park \cite{Chop2021} in a broad range of
the parameters. Wei and Zhang \cite{WeiZhang2023} later resolved that conjecture in full.
Together with the Hamiltonicity results above, these theorems reveal a recurring extremal
pattern: a graph with very large spectral radius can fail to contain the desired spanning
structure only because of a highly localized deficiency, often concentrated at one vertex.

Moving from individual spanning structures to a general framework, Liu and Ning
\cite{Liu2026} recently established an asymptotic spectral Tur\'an-type theorem for
sparse spanning graphs. They proved that if $F$ is an $n$-vertex graph with no isolated
vertices and maximum degree at most $\sqrt{n}/40$, then, for sufficiently large $n$, the
unique $F$-free graph of maximum adjacency spectral radius is obtained from $K_{n-1}$ by
adding one vertex adjacent to exactly $\delta(F)-1$ vertices of the clique. Applying their
theorem to a $K_k$-factor gives the correct extremal construction when $n$ is sufficiently
large in terms of $k$, but it does not provide the explicit range considered here. The aim of
the present paper is therefore to prove an exact result, including uniqueness, for every
$n=km$ with $m\ge 2k-1$.

To describe the candidate extremal graph, let
\[
  H_{n,k}=K_{k-2}\vee\bigl(K_{n-k+1}\cup K_1\bigr),
\]
where $\vee$ denotes the join of two vertex-disjoint graphs. Equivalently, start with a
clique $K_{n-1}$, add one new vertex, and join it to exactly $k-2$ vertices of the clique.
The new vertex lies in no copy of $K_k$, so $H_{n,k}$ cannot contain a $K_k$-factor.
On the other hand, the remaining $n-1$ vertices induce a complete graph. Thus the
obstruction is confined to one vertex and causes the smallest possible loss of adjacency
around that vertex, making $H_{n,k}$ the natural spectral extremal candidate.

Our main result confirms this intuition.

\begin{theorem}\label{thm:main}
Let $k\ge 3$ and $m\ge 2k-1$ be integers, and set $n=km$. If $G$ is an
$n$-vertex graph containing no $K_k$-factor, then
\[
  \rho(G)\le \rho(H_{n,k}).
\]
Moreover, equality holds if and only if
\[
  G\cong H_{n,k}=K_{k-2}\vee\bigl(K_{n-k+1}\cup K_1\bigr).
\]
\end{theorem}

We briefly outline the proof. A direct calculation using an equitable partition shows that
\[
  \rho(H_{n,k})>n-2.
\]
The main step is to prove that every $K_k$-factor-free graph $G$ with
$\delta(G)\ge k-1$ satisfies $\rho(G)\le n-2$. If $\rho(G)>n-2$, then a standard
spectral edge bound forces the complement $\overline{G}$ to have at most $n-2$ edges.
We then apply a decomposition lemma for such sparse complements, obtained using the
Hajnal--Szemer\'edi theorem. This produces a structural dichotomy: either $G$ contains
an independent set of size at least $m+1$, or some vertex of $G$ belongs to no copy of
$K_k$. The first possibility is ruled out by comparison with an appropriate complete split
graph. For the second, the neighborhood of the exceptional vertex is $K_{k-1}$-free; the
Motzkin--Straus inequality, together with a three-part decomposition of the Perron vector,
then yields the strict estimate $\rho(G)<n-2$.

It follows that a spectral-extremal graph must contain a vertex of degree at most $k-2$.
Since such a vertex cannot lie in a copy of $K_k$, adding any missing edge not incident
with it preserves the absence of a $K_k$-factor. Spectral maximality therefore forces all
other vertices to form a clique. The same edge-addition argument forces the exceptional
vertex to have degree exactly $k-2$, which identifies $H_{n,k}$ and proves the uniqueness
of the equality case. This argument illustrates a broader phenomenon in spectral factor
problems: an extremal obstruction to a spanning structure may be concentrated at one
exceptional vertex, while spectral maximality forces the rest of the graph to be complete.

The assumption $m\ge 2k-1$ enters through the decomposition of the sparse complement.
It is natural to ask whether the same extremal graph remains optimal, and uniquely so,
outside this range. In particular, the following problem remains open.

\begin{Problem}
For integers $k,m$ with $k\geq 3$ and $k\leq m\leq 2k-2$, is $H_{n,k}$ the unique extremal graph for the spectral radius among all $n$-vertex graphs with $n=km$ and no $K_k$-factor?
\end{Problem}

The remainder of the paper is organized as follows. In Section~\ref{Sec:2}, we establish the
decomposition lemma for sparse complements and the local spectral estimate used in the
proof. In Section~\ref{Sec:3}, we analyze the candidate graph $H_{n,k}$ and combine these ingredients
to prove Theorem~\ref{thm:main}.

\section{Preliminaries}\label{Sec:2}

For a graph $G$, let $V(G)$ and $E(G)$ denote its vertex and edge sets, respectively, and write $e(G)=|E(G)|$. The adjacency matrix of $G$ is denoted by $A(G)$,  and its spectral radius by $\rho(G)$. The minimum degree, maximum degree, independence number, and clique number of $G$ are denoted by $\delta(G)$, $\Delta(G)$, $\alpha(G)$, and $\omega(G)$, respectively. The complement of $G$ is denoted by $\overline{G}$.

\subsection{Structure Properties}

We first establish a decomposition lemma for sparse complements.

\begin{lemma}\label{cor:structural-dichotomy}
For integers $k\ge 3$ and $m\ge 2k-1$, let $G$ be a graph of order $km$ containing no $K_k$-factor. If
\[
e(\ol G)\le km-2,
\]
then at least one of the following two conclusions holds:
\begin{enumerate}[label=\textup{(\roman*)}]
\item $\alpha(G)\ge m+1$;
\item there exists a vertex $v\in V(G)$ contained in no copy of $K_k$.
\end{enumerate}
\end{lemma}

\begin{proof}
Suppose, to the contrary, that $\alpha(G)\le m$, and every vertex of $G$ is contained in a copy of $K_k$. It follows that every vertex of $\ol G$ is contained in an independent $k$-set.

\begin{claim}\label{Claim1} 
If $e(G)\le km-2$ and every vertex of $G$ is contained in an independent $k$-set, then $V(G)$ can be partitioned into $m$ independent $k$-sets.
\end{claim}

\begin{proof}

We construct the desired partition by repeatedly deleting independent
$k$-sets. For $1\le q\le m$, let $G_q$ denote the current induced
subgraph, so that
\[
|V(G_q)|=kq.
\]
We maintain the following two invariants:\\[2mm]
(a) $e(G_q)\le E_q:=km-2-\sum\limits_{i=q+1}^m i$;\\[1mm]
(b) every vertex of $G_q$ is contained in an independent $k$-set of $G_q$.\\[2mm]
Both invariants hold initially for $G_m=G$.

Suppose that $q\ge 2$ and $\Delta(G_q)\le q-1$.  The Hajnal--Szemer\'edi theorem \cite{KiKo2008} asserts that, for every positive integer $r$, every graph with maximum degree at most $r$ has an equitable coloring with $r+1$ colors, where a coloring is called equitable if the sizes of any two color classes differ by at most one. Applying this theorem to $G_q$ with $r=q-1$, we conclude that $G_q$ admits an equitable $q$-coloring. Since $|V(G_q)|=kq$, every color class must have size exactly $k$. Hence $G_q$ can be partitioned into $q$ independent $k$-sets. Together with the independent $k$-sets deleted earlier, this gives the desired partition of $V(G)$.

Suppose that $q\ge 2$ and $\Delta(G_q)\ge q$. Choose a maximum-degree vertex $v\in V(G_q)$. By the invariant, $v$ is contained in some independent $k$-set $T$. Put
\[
G_{q-1}=G_q-T.
\]
Since $d_{G_q}(v)\ge q$, deleting $T$ removes at least $q$ edges. Hence
\[
e(G_{q-1})\le e(G_q)-q\le E_q-q=E_{q-1},
\]
so the edge invariant is preserved. It remains to show that every vertex of $G_{q-1}$ is still contained in an independent $k$-set of $G_{q-1}$.

Suppose that there is a vertex $u\in V(G_{q-1})$ not contained in any independent $k$-set of $G_{q-1}$. Let
\[
A=V(G_{q-1})\setminus\bigl(\{u\}\cup N_{G_{q-1}}(u)\bigr),
\qquad s=|A|.
\]
If $G_{q-1}[A]$ contained an independent $(k-1)$-set, then together with $u$ it would form an independent $k$-set containing $u$, a contradiction. Thus
\[
\alpha(G_{q-1}[A])\le k-2.
\]
Set $r=k-2$, and let $t_r(s)=e(T_{s,r})$ be the number of edges in the complete $r$-partite Tur\'an graph on $s$ vertices. By the complementary form of Tur\'an's theorem,
\[
e(G_{q-1}[A])\ge \binom{s}{2}-t_r(s).
\]
Moreover,
\[
d_{G_{q-1}}(u)=k(q-1)-1-s.
\]
Therefore
\begin{align*}
e(G_{q-1})
&\ge d_{G_{q-1}}(u)+e(G_{q-1}[A])\\
&\ge k(q-1)-1-s+\binom{s}{2}-t_r(s).
\end{align*}
Define
\[
\phi_r(s)=\binom{s}{2}-t_r(s).
\]
Then $\phi_r(s)$ is the minimum number of edges in an $s$-vertex graph with independence number at most $r$. When $s$ is increased by one,
\[
\phi_r(s+1)-\phi_r(s)=\left\lfloor\frac{s}{r}\right\rfloor.
\]
It follows that $\phi_r(s)-s$ is decreasing for $s<r$, constant for $r\le s<2r$, and increasing for $s\ge 2r$. Hence
\[
\phi_r(s)-s\ge -r=-(k-2).
\]
Consequently,
\begin{equation}\label{eq:lower-bound-Fqminus1}
e(G_{q-1})\ge k(q-1)-1-(k-2)=kq-2k+1.
\end{equation}
On the other hand, by (a),
\[
e(G_{q-1})\le E_{q-1}=km-2-\sum_{i=q}^m i.
\]
We claim that
\begin{equation}\label{eq:key-gap}
E_{q-1}<kq-2k+1
\end{equation}
for all $2\le q\le m$. Let $t=m-q$. Then
\begin{align*}
kq-2k+1-E_{q-1}
&=kq-2k+1-km+2+\sum_{i=q}^m i\\
&=m-2k+3+t(m-k)-\frac{t(t+1)}2.
\end{align*}
This is a concave quadratic polynomial in $t$, so its minimum on the interval $0\le t\le m-2$ is attained at an endpoint. At the endpoints we have
\[
t=0:\qquad m-2k+3> 0,
\]
and
\[
t=m-2:\qquad m-2k+3+\frac{(m-2)(m-2k+1)}2> 0,
\]
where we use $m\ge 2k-1$. Thus \eqref{eq:key-gap} holds.

Equations \eqref{eq:lower-bound-Fqminus1} and \eqref{eq:key-gap} give a contradiction. Therefore every vertex of $G_{q-1}$ is contained in an independent $k$-set of $G_{q-1}$, and the invariants are preserved.

Each deletion step decreases $q$ by one, so the process must terminate.  Hence the process eventually reaches a stage at
which the Hajnal–Szemerédi theorem applies, or else reaches $q=1$. In the latter case, $G_1$ has $k$ vertices, and each of its vertices is contained in an independent $k$-set. Hence $V(G_1)$ itself is independent, and the process terminates. Combining the remaining partition with all deleted independent $k$-sets gives a partition of $V(G)$ into $m$ independent $k$-sets.
\end{proof}

Combining $e(\ol G)\le km-2$ and Claim \ref{Claim1}, the vertex set $V(\ol G)$ can be partitioned into $m$ independent $k$-sets. Therefore, $G$ contains a $K_k$-factor, a contradiction.
\end{proof}

\subsection{Spectral Properties}

We first recall the notion of an equitable partition and a standard spectral property of its quotient matrix. We then establish the key spectral estimate needed later.

\begin{definition}[\cite{cve1980}] \label{def:cve}
Given a graph $G$, the vertex partition $\Pi: V(G) = V_1 \cup V_2 \cup \cdots \cup V_k$ is called equitable if, for each $u \in V_i$, $|V_j \cap N_G(u)| = b_{ij}$ is a constant depending only on $i,j$ ($1 \leq i,j \leq k$). The matrix $B_\Pi = (b_{ij})$ is called the quotient matrix of $G$ with respect to $\Pi$.
\end{definition}

\begin{lemma}[\cite{cve1980}]\label{lem:cve} 
Let $\Pi: V(G) = V_1 \cup \cdots \cup V_k$ be an equitable partition of $G$ with quotient matrix $B_\Pi$. Then $\det(xI - B_\Pi) \mid \det(xI - A(G))$. Furthermore, if $G$ is connected, then the spectral radius of $B_\Pi$ equals the spectral radius of $G$.
\end{lemma}

\begin{lemma}\label{lem:local-spectral-exclusion}
For integers $k\ge 3$ and $m\ge 2k-1$, let $G$ be a graph of order $n=km$ satisfying $\delta(G)\ge k-1$. If there exists a vertex $v$ contained in no copy of $K_k$, then
\[
\rho(G)<n-2.
\]
\end{lemma}

\begin{proof}
Let
\[
D=N_G(v),\qquad
S=V(G)\setminus\bigl(D\cup\{v\}\bigr),
\]
and put
\[
d=|D|,\qquad
s=|S|=n-d-1,\qquad
r=k-2.
\]
Since $\delta(G)\ge k-1$, we have $d\ge k-1$. Since $v$ is contained in no copy of $K_k$, the graph $G[D]$ contains no $K_{k-1}$. Hence 
\begin{equation}
\omega(G[D])\le k-2=r.
\label{eq:clique-bound}
\end{equation}

Let $\bm x=(x_u)_{u\in V(G)}$ be a nonnegative unit eigenvector of $G$ corresponding to $\rho(G)$. Define
\[
a=x_v,
\qquad
b=\left(\sum_{u\in D}x_u^2\right)^{1/2},
\qquad
c=\left(\sum_{u\in S}x_u^2\right)^{1/2}.
\]
Then $a^2+b^2+c^2=1$. We estimate the Rayleigh quotient by splitting the edges according to the
partition
\[
V(G)=\{v\}\cup D\cup S.
\]
First consider the edges inside $D$. Let
\[
W=\sum_{u\in D}x_u.
\]
If $W=0$, then their contribution is zero. Suppose that $W>0$, and set
\[
y_u=\frac{x_u}{W},\qquad u\in D.
\]
Then $y_u\ge0$ and
\[
\sum_{u\in D}y_u=1.
\]
By \eqref{eq:clique-bound} and the Motzkin--Straus theorem \cite{MS1965},
\[
\sum_{uw\in E(G[D])}y_uy_w
\le
\frac12\left(1-\frac1r\right).
\]
Multiplying by $2W^2$ and then applying the Cauchy--Schwarz inequality, we
obtain
\begin{align}
2\sum_{uw\in E(G[D])}x_ux_w
&\le
\left(1-\frac1r\right)
\left(\sum_{u\in D}x_u\right)^2 \notag\\
&\le
\left(1-\frac1r\right)db^2.
\label{4}
\end{align}
Since every vertex of $D$ is adjacent to $v$, we have
\begin{align}
2\sum_{u\in D}x_vx_u\le 2a\sqrt d\,b.
\end{align}
Moreover,
\begin{align}
2\sum_{uw\in \E(D,S)}x_ux_w
\le 2\left(\sum_{u\in D}x_u\right)\left(\sum_{w\in S}x_w\right)
\le 2\sqrt{ds}\,bc,
\end{align}
and
\begin{align}
2\sum_{uw\in \E(G[S])}x_ux_w
\le 2\sum_{\{u,w\}\subseteq S}x_ux_w
\le (s-1)c^2.
\label{7}
\end{align}
Combining Inequalities \eqref{4}--\eqref{7}, we obtain
\[
\rho(G)\le \left(1-\frac1r\right)db^2+2a\sqrt d\,b+2\sqrt{ds}\,bc+(s-1)c^2.
\]
Equivalently,
\[
\rho(G)\le
\begin{pmatrix}a&b&c\end{pmatrix}
B
\begin{pmatrix}a\\ b\\ c\end{pmatrix},
\]
where
\[
B=\begin{pmatrix}
0&\sqrt d&0\\
\sqrt d&\left(1-\frac1r\right)d&\sqrt{ds}\\
0&\sqrt{ds}&s-1
\end{pmatrix}.
\]
Since $a^2+b^2+c^2=1$, it follows that $\rho(G)\le \rho(B)$. 

We dintinguish two cases.

\begin{case}
$s=0$
\end{case}

In this case, we have $d=n-1$, and the matrix reduces to
\[
B=\begin{pmatrix}
0&\sqrt d\\
\sqrt d&\left(1-\frac1r\right)d
\end{pmatrix}.
\]
Its characteristic polynomial is
\[
\phi(x)=x^2-\left(1-\frac1r\right)dx-d.
\]
Substituting $x=d-1=n-2$ gives
\[
\phi(d-1)=\frac{d^2-(2r+1)d+r}{r}.
\]
Since $d=n-1\ge k(2k-1)-1$, in particular $d\ge 2r+1$, and hence
\[
d^2-(2r+1)d+r>0.
\]
Also $\phi(0)=-d<0$. Thus the positive root of $\phi$ is strictly smaller than $d-1=n-2$, and so 
\[
\rho(G)\le\rho(B)<n-2.
\]

\begin{case}
$s\ge1$
\end{case}

Set
\[
\theta=n-2=d+s-1,
\qquad
\mu=1-\frac1r.
\]
We construct a positive vector $z$ such that $Bz<\theta z$. First note that
\[
\Delta_0:=\frac dr-1-\frac d\theta>0.
\]
Indeed, this is equivalent to
\[
\theta>\frac{dr}{d-r}.
\]
Since $d\ge k-1=r+1$, the function $dr/(d-r)$ is decreasing in $d$, and hence
\[
\frac{dr}{d-r}\le r(r+1).
\]
On the other hand,
\[
\theta=n-2\ge k(2k-1)-2=(r+2)(2r+3)-2>r(r+1).
\]
Thus $\Delta_0>0$. Choose $\eta>0$ sufficiently small so that
\[
\eta\left(s+\frac d\theta\right)<\Delta_0.
\]
Set
\[
z=\begin{pmatrix}
(1+\eta)\dfrac{\sqrt d}{\theta}\\[0.6em]
1\\[0.4em]
(1+\eta)\sqrt{\dfrac{s}{d}}
\end{pmatrix}.
\]
Clearly $z>0$. We verify the three coordinates. For the first coordinate,
\[
(Bz)_0=\sqrt d<(1+\eta)\sqrt d=\theta z_0.
\]
For the second coordinate,
\[
(Bz)_1=\sqrt d\,z_0+\mu dz_1+\sqrt{ds}\,z_2
=\frac{d(1+\eta)}\theta+\mu d+s(1+\eta).
\]
Thus $(Bz)_1<\theta z_1=\theta$ is equivalent to
\[
\frac{d(1+\eta)}\theta+\eta s<\frac dr-1.
\]
The left-hand side equals
\[
\frac d\theta+\eta\left(\frac d\theta+s\right),
\]
which is strictly smaller than $d/r-1$ by the choice of $\eta$. 
For the third coordinate,
\begin{align*}
\theta z_2-(Bz)_2
&=(d+s-1)(1+\eta)\sqrt{\frac sd}-\sqrt{ds}-(s-1)(1+\eta)\sqrt{\frac sd}\\
&=d(1+\eta)\sqrt{\frac sd}-\sqrt{ds}
=\eta\sqrt{ds}>0.
\end{align*}
Therefore $Bz<\theta z$. 
Since $z>0$ and $Bz<\theta z$ coordinatewise, the Collatz--Wielandt inequality yields
\[
\rho(B)
\le
\max_i \frac{(Bz)_i}{z_i}
<
\theta.
\]
Therefore,
\[
\rho(G)\le \rho(B)<\theta=n-2.
\]
\end{proof}

\section{Proof of the main result}\label{Sec:3}

\begin{proof}[Proof of Theorem~\ref{thm:main}]
Let $G$ be a graph of maximum spectral radius among all graphs of order $n$ with no $K_k$-factor. 

\begin{claim}\label{lem:lower-spectral}
$\rho(G)>n-2$.
\end{claim}

\begin{proof}
Partition $V(H_{n,k})$ into three parts
\[
V_0=\{v\},\qquad |V_1|=k-2,\qquad |V_2|=n-k+1,
\]
where $V_1=N(v)$ and $V_1\cup V_2$ induces $K_{n-1}$. This is an equitable partition, and its quotient matrix is
\[
B=\begin{pmatrix}
0&k-2&0\\
1&k-3&n-k+1\\
0&k-2&n-k
\end{pmatrix}.
\]
Therefore $\rho(H_{n,k})=\rho(B)$. A direct calculation gives
\[
\phi_{n,k}(x)=\det(xI-B)=x^3-(n-3)x^2-(n+k-4)x+(k-2)(n-k).
\]
Substituting $x=n-2$, we obtain
\[
\phi_{n,k}(n-2)=-(k-2)^2<0.
\]
Since $\phi_{n,k}(x)\to +\infty$ as $x\to +\infty$, the polynomial $\phi_{n,k}$ has a real root greater than $n-2$. This root is an eigenvalue of $B$, and hence it is at most the spectral radius of $B$. Consequently,
\[
\rho(H_{n,k})=\rho(B)>n-2.
\]
Since $H_{n,k}$ has no $K_k$-factor and satisfies
$\rho(G)\ge \rho(H_{n,k})$, we have \[\rho(G)>n-2.\]
\end{proof}

\begin{claim}\label{connectness}
$G$ is connected.
\end{claim}

\begin{proof}
Suppose, to the contrary, that $G$ is disconnected. Let $C$ be a component of $G$ such that
$\rho(C)=\rho(G)$. We can obtain a graph $G+xy$ from $G$ by adding an edge $xy$, where $x\in V(C)$ and a vertex $y\in V(G-C)$. Since $k\ge 3$ and $xy$ is a bridge of the component of $G+xy$ containing $C$, the single added edge $xy$ is contained in no copy of $K_k$. Therefore $G+xy$ still has no $K_k$-factor.
Let $C'$ be the component of $G+xy$ containing $C$. Note that $C$ is a proper subgraph of the connected graph $C'$. By the strict monotonicity of the spectral radius \cite{SD2015},
\[
\rho(G+xy)\ge\rho(C')>\rho(C)=\rho(G),
\]
a contradiction. Therefore, $G$ is connected.
\end{proof}

\begin{claim}\label{lem:core-spectral}
If $\delta(G)\ge k-1$, then $\rho(G)\le n-2$.
\end{claim}

\begin{proof}
Suppose, to the contrary, that $\rho(G)>n-2$. Since $\delta(G)\ge k-1\ge 2$, the graph $G$ has no isolated vertices. By  Hong-type spectral radius inequality \cite{Hong1988}, we have
\[
(n-2)^2<\rho(G)^2\le 2e(G)-n+1.
\]
Thus
\[
2e(G)>n^2-3n+3.
\]
Since $e(G)$ is an integer, we have
\[
e(G)\ge \frac{n^2-3n+4}{2}=\binom{n-1}{2}+1.
\]
Therefore
\[
e(\ol G)=\binom n2-e(G)\le \binom n2-\binom{n-1}{2}-1=n-2=km-2.
\]
By Lemma~\ref{cor:structural-dichotomy}, at least one of the following two cases occurs.

First, suppose that $\alpha(G)\ge m+1$. Choose an independent set  of size $m+1$ in $V(G)$. Then
\[
G\subseteq J_{m,k}=K_{(k-1)m-1}\vv \ol{K_{m+1}}.
\]
With respect to the partition into the complete part and the independent part, the quotient matrix of $J_{m,k}$ is
\[
B_J=
\begin{pmatrix}
(k-1)m-2&m+1\\
(k-1)m-1&0
\end{pmatrix}.
\]
Its characteristic polynomial is
\[
\phi_J(x)=x^2-\bigl((k-1)m-2\bigr)x-\bigl((k-1)m-1\bigr)(m+1).
\]
Substituting $x=n-2=km-2$, we obtain
\[
\phi_J(n-2)=m^2-km+1.
\]
Since $m\ge 2k-1\ge k$, we have $m^2-km+1\ge 1>0$. Since $\phi_J(0)<0$, the positive root of $\phi_J$ is strictly smaller than $n-2$. Hence 
\[
\rho(G)\le \rho(J_{m,k})<n-2,
\]
a contradiction.

Second, suppose that there exists a vertex $v$ contained in no copy of $K_k$. Then Lemma~\ref{lem:local-spectral-exclusion} gives
\[
\rho(G)<n-2,
\]
again a contradiction.

Therefore, $\rho(G)\le n-2$.
\end{proof}

By Claims~\ref{lem:lower-spectral} and \ref{lem:core-spectral}, we have $\delta(G)\le k-2$. Choose a vertex $v$ of minimum degree, i.e., $d_G(v)\le k-2$. So $v$ is contained in no copy of $K_k$.

We next prove that $G-v$ must be complete. Suppose that there exists an edge $xy\notin E(G-v)$. We can obtain the graph $G'$ by adding the edge $xy$ to $G$. Since the edge $xy$ is not incident with $v$, the degree of $v$ remains at most $k-2$, and hence $v$ is still contained in no copy of $K_k$. Therefore, $G'$ still has no $K_k$-factor. By Claim \ref{connectness} and the strict monotonicity of the spectral radius \cite{SD2015}, we have $\rho(G)<\rho(G'),$ a contradiction. Thus
\[
G-v\cong K_{n-1}.
\]

It remains to prove that $d_G(v)=k-2$. If $d_G(v)\le k-3$, then there exists a non-neighbor $u$ of $v$ in $G-v$. We can obtain the graph $G''$ by adding the edge $uv$. The degree of $v$ in $G''$ is still at most $k-2$, so $v$ is still contained in no copy of $K_k$. Thus $G''$ has no $K_k$-factor. By Claim \ref{connectness} and the strict monotonicity of the spectral radius \cite{SD2015}, we have $\rho(G)<\rho(G''),$ again a contradiction. Therefore
\[
d_G(v)=k-2.
\]
Consequently, $G$ consists of a copy of $K_{n-1}$ together with one special vertex adjacent to exactly $k-2$ vertices of the clique. Hence
\[
G\cong K_{k-2}\vv\bigl(K_{n-k+1}\cup K_1\bigr)=H_{n,k}.
\]
This proves the uniqueness of the extremal graph. In particular,
\[
\rho(G)\le \rho(H_{n,k}),
\]
and equality holds if and only if $G\cong H_{n,k}$.
\end{proof}

\end{document}